\documentclass{amsart}
\usepackage[english]{babel}
\usepackage[latin1]{inputenc}
\usepackage[dvips,final]{graphics}
\usepackage{amsmath,amsfonts,amssymb,amsthm,amscd,array,stmaryrd,mathrsfs, mathdots, epigraph}
\usepackage{pstricks}
 \usepackage[all]{xy}
 \usepackage{url}
\usepackage{multirow, blkarray}
\usepackage{booktabs}
\usepackage{textcomp}
 \usepackage[final]{epsfig}
 \usepackage{color}
\theoremstyle{plain}
\newtheorem{thm}{Theorem}
\newtheorem{lem}{Lemma}[section]

\newtheorem{prop}[lem]{Proposition}

\theoremstyle{definition}

\newcommand{\R}{\mathbb{R}}
\newcommand{\Z}{\mathbb{Z}}
\newcommand{\C}{\mathbb{C}}
\newcommand{\bbH}{\mathbb{H}}
\newcommand{\bbF}{\mathbb{F}}

\newcommand{\rCl}{\mathrm{Cl}}

\newcommand{\Mat}{\textup{Mat}}

\def\a{\alpha}

\def\d{\delta}

\begin{document}

\title[Old problems become one]{Real Clifford algebras and quadratic forms over $\bbF_2$:\\
two old problems become one}

\author{Valentin Ovsienko}

\address{
Valentin Ovsienko,
CNRS,
Laboratoire de Math\'ematiques 
U.F.R. Sciences Exactes et Naturelles 
Moulin de la Housse - BP 1039 
51687 REIMS cedex 2,
France}

\email{valentin.ovsienko@univ-reims.fr}




\maketitle

{\hfill {\it Importance of establishing connections}}

{\hfill {\it between different branches of scientific enquiry; }}

{\hfill {\it such steps are of the nature of revolutions in science.}}

\medskip
{\hfill W.K. Clifford}
\bigskip


Clifford algebras are 
important for many areas of pure and applied mathematics and physics.
Real Clifford algebras were introduced~\cite{Cliff} by the
English mathematician and philosopher William Kingdon Clifford in 1878.
Clifford called them ``geometric algebras''.
They generalize 
complex numbers and quaternions,
as well as multi-dimensional Grassmann algebras.
Real Clifford algebras and their representations appeared
independently in the work of Hurwitz on square identities~\cite{Hur1},
and still play an essential role in this subject.
Application of real Clifford algebras in topology was initiated in~\cite{ABS}.

Every real Clifford algebra is isomorphic to one of
the algebras of real, complex or quaternionic matrices
whose size is a power of~$2$, 
or their ``double copies'', i.e., the direct sum of one of these matrix algebras
with itself.
More precisely, one has the following theorem
of Chevalley~\cite{Che}.

\begin{enumerate}
\item
There are exactly two real Clifford algebras of dimension
$2^n$ for $n=2k$, namely $\Mat(2^{k},\R)$
and $\Mat(2^{k-1},\bbH)$.

\item
There are exactly three real Clifford algebras of dimension
$2^n$ for $n=2k+1$: the algebra $\Mat(2^{k},\C)$,
which is the only simple algebra in this case, and
two double copies of the above algebras of real and quaternionic matrices.
\end{enumerate}

\noindent
The usual modern way to obtain this classification uses periodicity results,
among which the Bott periodicity is the most famous one;
the reader is invited to consult the book~\cite{LM},
and an online course~\cite{FOF}.

The problem of classification of quadratic forms
over the field of two elements $\bbF_2$
was solved by Dickson~\cite{Dick}.
Recall that the vector space $\bbF_2^n$ consists in
$n$-tuples $x=(x_1,\ldots,x_n)$, where $x_i\in\{0,1\}$.
The classification result can be formulated as follows.

\begin{enumerate}
\item[(1')]
There are exactly two equivalence classes of 
non-degenerate\footnote{
The terminology will be clarified in Section~\ref{DS}.} 
quadratic forms on $\bbF_2^{2k}$,
these classes can be represented by the forms:
$$
\begin{array}{rcl}
{\bf q}_0(x)&=&x_1x_2+x_3x_4+\cdots+x_{2k-1}x_{2k},
\\[2pt]
{\bf q}_1(x)&=&x_1x_2+x_3x_4+\cdots+x_{2k-1}x_{2k}+x_{1}+x_{2},
\end{array}
$$
where the summation is performed modulo~$2$.
Note that the form ${\bf q}_1$ can also be written as a homogeneous quadratic form,
since $x_i^2=x_i$.

\item[(2')]
The form
$$
{\bf q}_2(x)=x_1x_2+x_3x_4+\cdots+x_{2k-1}x_{2k}+x_{2k+1},
$$
is the only regular
quadratic form on $\bbF_2^{2k+1}$.
There are exactly three equivalence classes of quadratic forms of rank $2k$
on $\bbF_2^{2k+1}$, those of
${\bf q}_0,{\bf q}_1$, trivially extended to $\bbF_2^{2k+1}$, and~${\bf q}_2$.
\end{enumerate}

Non-equivalence of the forms ${\bf q}_0$ and ${\bf q}_1$
follows from the {\it Arf invariant},
defined as follows:
$\mathrm{Arf}({\bf q})=1$ if the number of points $x\in\bbF_2^n$
such that ${\bf q}(x)=1$ is greater than the number of points at which ${\bf q}(x)=0$;
otherwise, $\mathrm{Arf}({\bf q})=0$.
An easy computation then shows that 
$$
\mathrm{Arf}({\bf q}_0)=0,
\qquad
\mathrm{Arf}({\bf q}_1)=1.
$$
The proof of the above theorem is very simple, see Appendix~1.

The goal of this note is to explain that the noticeable similarity 
of the above classification theorems is not a coincidence:
the problems are, indeed, equivalent.
This equivalence is quite surprising, and
can even be misunderstood.
Every Clifford algebra
is associated with a quadratic form, and every quadratic form
defines a Clifford algebra (the algebras and the forms are over the same ground field).
This is a classical and tautological relation,
which is not an equivalence as we will see.
Here, we compare Clifford algebras over $\R$ and quadratic forms over $\bbF_2$.

Quadratic forms over $\bbF_2$ are also very useful in topology, see, e.g.,~\cite{Kir}.
Can this be a reason or a consequence for equivalency
of two theories?
In particular, it is amusing to know that the difference between
real and quaternionic matrices is measured by the Arf invariant.
This could perhaps provide Clifford with some material as philosopher...

\section{Definition of real Clifford algebras}

The simplest definition of
a real Clifford algebra is the original definition of Clifford~\cite{Cliff}.

The real Clifford algebra $\rCl_{p,q}$
is the associative algebra with unit $\bf1$
and $n=p+q$ generators
$\iota_1,\ldots,\iota_n$ that anticommute:
\begin{equation}
\label{ClRel}
\iota_i\iota_j=-\iota_j\iota_i,\quad{}i\not=j,
\end{equation}
and square to $\bf1$ or $-\bf1$:
\begin{equation}
\label{ClRelBis}
\iota_i^2=\left\{
\begin{array}{rl}
\bf1,&1\leq{}i\leq{}p\\[2pt]
-\bf1,&p<i\leq{}n.
\end{array}
\right.
\end{equation}
The pair of numbers $(p,q)$ is called the signature.
The monomials 
$$
\iota_I=\iota_{i_1}\cdots{}\iota_{i_k},
$$
where $1\leq{}i_1<\ldots<i_k\leq{}n$, and $I=\{i_1,\ldots,i_k\}$
 form a basis of $\rCl_{p,q}$ 
so that, $\dim\rCl_{p,q}=2^n$.

A more general definition of Clifford algebras
is as follows.
Given a vector space $V$ over a field~$\bbF$,
and a quadratic form $Q:V\to\bbF$,
the corresponding Clifford algebra is the quotient
of the tensor algebra $T(V)$ by the bilateral ideal
generated by the elements of the form 
$$
v\otimes{}v-Q(v)\bf1,
$$
for all $v\in{}V$, and where $\bf1$ is the unit.

Choose a basis $\{\iota_1,\ldots,\iota_n\}$ of $V$, so that the Clifford algebra
is generated by $\iota_1,\ldots,\iota_n$, with relations
$$
\iota_i^2=Q(\iota_i){\bf1},
\qquad
\iota_i\iota_j+\iota_j\iota_i=B(\iota_i,\iota_j){\bf1},
$$
where $B$ is the {\it polar bilinear form} defined by 
$$
B(u,v):=Q(u+v)-Q(u)-Q(v).
$$

We consider only the case where $\bbF=\R$
and the bilinear form $B$ is non-degenerate.
Choosing a basis in which $B(\iota_i,\iota_j)=\pm\d_{i,j}$,
one then goes back to the above Clifford definition.
Note that, although real Clifford algebras are parametrized by two numbers $(p,q)$,
i.e., the signature of the quadratic form,
many of them are isomorphic.
Therefore, there is no equivalence between real quadratic forms and 
real Clifford algebras.

\section{Quadratic forms over $\bbF_2$}\label{DS}

A quadratic form ${\bf q}:\bbF_2^n\to\bbF_2$ can be written in coordinates as follows:
$$
{\bf q}(x)=\sum_{1\leq{}i\leq{}j\leq{}n}q_{ij}\,x_ix_j,
$$
where the coefficients $q_{ij}\in\{0,1\}$,
and where the summation is modulo~$2$.
Note that the ``diagonal terms'' of ${\bf q}$:
$$
\sum_{1\leq{}i\leq{}n}q_{ii}\,x^2_i
$$
constitute its {\it linear part},
since $x_i^2=x_i$.

A quadratic form on $\bbF_2^n$ also defines 
the polar bilinear form\footnote{
Note that in characteristic $2$, there is no difference between the ``$+$'' and ``$-$'' signs.}
$$
{\bf b}(x,y):={\bf q}(x+y)+{\bf q}(x)+{\bf q}(y),
$$
which is {\it alternating}, i.e., ${\bf b}(x,x)={\bf q}(0)=0$.
This implies that
its rank is always even.

The relation between quadratic forms and 
the corresponding polar bilinear alternating forms is
that the polar form ${\bf b}$ ``forgets'' about the linear part of ${\bf q}$.
Two quadratic forms ${\bf q}$ and ${\bf q}'$ on~$\bbF_2^n$ correspond to the same
polar form
if and only if ${\bf q}+{\bf q}'$ is a linear function.
This is the main difference with the real case where the quadratic form
can be reconstructed from the corresponding polar bilinear form.

The {\it rank} of a quadratic form on $\bbF_2^n$ is
defined as the rank of the corresponding 
polar form.
Therefore, a quadratic form over $\bbF_2$ can be {\it non-degenerate}
(i.e., of full rank)
only if $n=2k$.
If $n=2k+1$, and a quadratic form has rank $n-1$,
then it makes sense to ask if it is
{\it regular}, i.e., not equivalent to a form written in less that $n$ variables.

\section{Equivalence}
We will now explain the equivalence of the theories of
real Clifford algebras and quadratic forms on~$\bbF_2^n$.
The following crucial idea was suggested by Albuquerque and Majid~\cite{AM1}.

Consider the vector space $\R[\bbF_2^n]\simeq\R^{2^n}$ with
natural basis $e_x$, where $x\in\bbF_2^n$.
Since $\bbF_2^n$ is an abelian group,
the space $\R[\bbF_2^n]$
has a structure of commutative algebra defined by
$$
e_xe_y=e_{x+y}.
$$
Identify the basis of $\rCl_{p,q}$ 
with the basis of $\R[\bbF_2^n]$ by
\begin{equation}
\label{IdentiF}
\iota_I\longleftrightarrow
e_{x_I},
\end{equation}
where $x_I=(x_1,\ldots,x_n)$, such that $x_i=1$, iff $i\in{}I$.
The algebra $\rCl_{p,q}$ is thus identified with~$\R[\bbF_2^n]$,
as a vector space.
However, the product in $\rCl_{p,q}$ and in~$\R[\bbF_2^n]$ differs by a sign.

Consider a new product in~$\R[\bbF_2^n]$:
$$
e_x\cdot{}e_y:=(-1)^{f(x,y)}e_{x+y},
$$
where $f:\bbF_2^n\times\bbF_2^n\to\bbF_2$ is a function of two arguments.
This new structure is an algebra called a {\it twisted group algebra},
we denote it $(\R[\bbF_2^n],f)$.

\begin{prop}~\cite{AM1}.
\label{SuPr}
The algebra $\rCl_{p,q}$ is isomorphic to $(\R[\bbF_2^n],f)$
where $f$ is the bilinear form
$$
f(x,y)=\sum_{1\leq{}j<i\leq{}n}x_iy_j+\sum_{p+1\leq{}i\leq{}n}x_iy_i.
$$
\end{prop}

\begin{proof}
Let us first check that the generators
$e_{{\bf x}_i}$, where
$$
{\bf x}_i=(0\ldots0\,1\,0\ldots0)
$$ 
with $1$ is at $i$th position,
satisfy the relations~(\ref{ClRel}) and~(\ref{ClRelBis}).
Indeed, $f({\bf x}_i,{\bf x}_j)=0$ and $f({\bf x}_j,{\bf x}_i)=1$, provided $i<j$,
so that the generators anticommute.
Furthermore, since $f({\bf x}_i,{\bf x}_i)=1$ for $i>p$, 
the generators~$e_{{\bf x}_i}$ square to $-\bf1$ for $i>p$ and to $\bf1$ for $i\leq{}p$.

The algebra $(\R[\bbF_2^n],f)$ is associative,
as readily follows from the fact that $f$ is bilinear.
Therefore, the map (\ref{IdentiF}) is a homomorphism of $(\R[\bbF_2^n],f)$ to~$\rCl_{p,q}$
with trivial kernel.
\end{proof}

The above realization of real Clifford algebras as twisted group algebras
was used in~\cite{AM1} to recover structural results, such as periodicities.

Let us go one step further and address the question of isomorphism
of twisted group algebras with bilinear twisting functions.
Define the ``diagonal''  quadratic form
$$
\a(x):=f(x,x).
$$
It turns out that the algebra is completely determined by $\a$.

\begin{thm}
\label{EqClTh}
Two twisted group algebras $(\R[\bbF_2^n],f_1)$ and $(\R[\bbF_2^n],f_2)$ with bilinear 
functions~$f_1$ and~$f_2$, are isomorphic
if and only if the corresponding quadratic forms,~$\a_1$ and~$\a_2$,
are equivalent.
\end{thm}

Our proof consists in two parts.
We first prove the ``if'' part, while
the converse statement,
based on classification of quadratic forms, will be proved in the end of Section~\ref{FFTA}.

Consider a twisted group algebra $(\R[\bbF_2^n],f)$ with bilinear $f$,
and two generators, $e_i,e_j$.
Their commutation relation is determined by the value of
$f({\bf x}_i,{\bf x}_j)+f({\bf x}_j,{\bf x}_i)$.
It turns out that the symmetrization of $f$
coincides with the polarization of $\a$, i.e.,
$$
f(x,y)+f(y,x)=\a(x+y)+\a(x)+\a(y),
$$
for all $x,y\in\bbF_2^n$.
Indeed, it suffices to check this for every monomial $x_iy_j$.
Therefore, the quadratic form $\a$ completely determines the relations between the generators.

We proved that if~$\a_1=\a_2$, then the algebras are isomorphic.
Since, a twisted group algebra does not change under coordinate transformations 
of $\bbF_2^n$, the ``if'' part follows.

To prove the ``only if'' part of the theorem,
we will use Dickson's classification of quadratic forms on $\bbF_2^n$,
and show that non-equivalent quadratic forms correspond to non-isomorphic algebras.

\section{From quadratic forms to algebras}\label{FFTA}

Let us consider the quadratic forms 
${\bf q}_0,{\bf q}_1$ and~${\bf q}_2$,
and show that the corresponding twisted group algebras on $\R[\bbF_2^n]$
are precisely the matrix algebras
appearing in the Chevalley classification of real Clifford algebras.

Consider first the $2$-dimensional case, and
the quadratic form ${\bf q}_0=x_1x_2$ on $\bbF_2^2$,
The algebra~$\R[\bbF_2^2]$ has
two generators,
$e_1$ and $e_2$.
Since the polarization of ${\bf q}_0$ is the $2$-form
$$
{\bf b}(x,y)=x_1y_2+x_2y_1,
$$
the generators anticommute,
and since ${\bf q}_0({\bf x}_1)={\bf q}_0({\bf x}_2)=0$,
one has $e_1^2=e_2^2=\bf1$.
This algebra is isomorphic to the algebra of real $2\times2$ matrices.
Indeed, the following generators of $\Mat(2,\R)$:
$$
e_1=
\left(
\begin{array}{cc}
1&0\\
0&-1
\end{array}
\right),
\qquad
e_2=
\left(
\begin{array}{cc}
0&1\\
1&0
\end{array}
\right)
$$
satisfy the above relations.

The quadratic form ${\bf q}_1=x_1x_2+x_1+x_2$ on $\bbF_2^2$
corresponds to the case
$e_1^2=e_2^2=-\bf1$, since ${\bf q}_1({\bf x}_1)={\bf q}_1({\bf x}_2)=1$.
Thus $e_1$ and $e_2$ generate the algebra of quaternions $\bbH$.

(1)
Consider now the form ${\bf q}_0$ on $\bbF_2^{2k}$.
The generators of the corresponding algebra structure on~$\R[\bbF_2^{2k}]$ 
in each $2$-dimensional block:
$\{e_1,e_2\},\{e_3,e_4\},\ldots$ form a copy of $\Mat(2,\R)$.
The generators from different blocks commute,
so that the algebra on $\R[\bbF_2^{2k}]$
is just the tensor product of $k$ copies:
$$
\Mat(2,\R)\otimes\Mat(2,\R)\otimes\cdots\otimes\Mat(2,\R)
\simeq\Mat(2^k,\R).
$$
Similarly, the algebra on~$\R[\bbF_2^{2k}]$ corresponding to the form ${\bf q}_1$
is
$$
\Mat(2,\R)^{\otimes{}(k-1)}\otimes\bbH
\simeq\Mat(2^{k-1},\bbH).
$$
The forms ${\bf q}_0$ and ${\bf q}_1$ thus correspond
to the algebras of real and quaternionic matrices, respectively.

(2)
The form ${\bf q}_2$ on $\bbF_2^{2k+1}$
corresponds to the algebra
$$
\Mat(2,\R)^{\otimes{}k}\otimes\C
\simeq\Mat(2^{k},\C),
$$
since the last generator $e_{2k+1}$
commutes with $e_i$ for $i\leq{}2k$ and squares to $-\bf1$, 
and therefore generates the
algebra of complex numbers.

Finally, the forms 
${\bf q}_0$ and ${\bf q}_1$ trivially extended to $\bbF_2^{2k+1}$
correspond to the algebras 
$$
\Mat(2^k,\R)\otimes\R^2
\quad\hbox{and}\quad\Mat(2^{k-1},\bbH)\otimes\R^2,
$$
respectively.
These are just the double copies of the above matrix algebras.

We are ready to complete the proof of Theorem~\ref{EqClTh}.
Consider a twisted group algebra $(\R[\bbF_2^n],f)$ with bilinear function $f$,
and the corresponding quadratic form $\a$.
By Dickson's theorem, every quadratic form on $\bbF_2^n$
of rank $2k$ is equivalent to
${\bf q}_0,{\bf q}_1$, or~${\bf q}_2$.
We have just proved that $(\R[\bbF_2^n],f)$ 
must be isomorphic to a direct sum of
several copies of $\Mat(2^k,\R)$, $\Mat(2^{k-1},\bbH)$, or
$\Mat(2^{k},\C)$, respectively.
Sense these three algebras are obviously not isomorphic
to each other, we proved that non-equivalent quadratic forms
correspond to non-isomorphic algebras.

\section{Comments}
The above linear algebra considerations hide the cohomological
nature of Theorem~\ref{EqClTh}.
A twisted group algebra
is associative if and only if the function $f$ satisfies
$$
(\d{}f)(x,y,z):=
f(x,y)+f(x,y+z)+f(x+y,z)+f(y,z)=0,
$$
for all $x,y,z\in\bbF_2^n$.
Such a function $f$ is called a $2$-cocycle on the abelian group~$\Z_2^n$
(it is convenient to use this notation for $\bbF_2^n$ considered only as an abelian group,
 and not as a vector space).
This condition is obviously satisfied if $f$ is bilinear,
but bilinearity is, of course, not necessary.
The above condition is also always satisfied for the functions
$f$ of the form
$$
f(x,y)=g(x+y)+g(x)+g(y),
$$
where $g$ is an arbitrary function
of one argument.
Then $f$ is called a trivial cocycle, or a coboundary.
For instance, the polar form $\bf b$ is the coboundary of $\bf q$.
A statement closely related to Theorem~\ref{EqClTh}, affirms that,
for an arbitrary $2$-cocycle $f$ (not necessarily bilinear),
the diagonal function $f(x,x)$
must be a quadratic form, and this form
determines the cohomology class of $f$ 
(for more details, see~\cite{MGO}).

A more general class of twisted group algebras $(\R[\bbF_2^n],f)$,
that contains the classical algebra of octonions,
was considered in~\cite{MGO}.
These are algebras where $f$ is not necessarily a $2$-cocycle, but
$\d{}f$ is a symmetric function of three arguments.
In this case, $\a(x):=f(x,x)$ must be a cubic form on $\bbF_2^n$.
Two such algebras are proved to be isomorphic
(at least as $\Z_2^n$-graded algebras) if and only if the corresponding
cubic forms are equivalent.
Note that classification of cubic forms on $\bbF_2^n$ is a difficult unsolved problem.

Quadratic forms over $\bbF_2$ have been recently used in~\cite{RE}
to classify gradings of simple real algebras by abelian groups.
The quadratic forms 
${\bf q}_0,{\bf q}_1$ and~${\bf q}_2$
appear explicitly in this classification.
This result is of course closely related to real Clifford algebras,
see in particular Remark 17 of~\cite{RE}.

Clifford algebras were considered as superalgebras,
i.e., $\Z_2$-graded algebras, already in~\cite{ABS}.
Thanks to Proposition~\ref{SuPr}, Clifford algebras with $n$
generators can be understood as
{\it graded-commutative} algebras over $\Z_2^n$
(see also~\cite{Sim} for a classification).
This viewpoint was recently used to revisit the old classical problem
due to Cayley, of developing linear algebra with coefficients in Clifford algebras,
see~\cite{CM} and references therein.
In particular, it offers a new understanding of
the Dieudonn\'e determinant
of quaternionic matrices.

\bigskip

{\bf Acknowledgements}.
I am pleased to thank Alberto Elduque and Sophie Morier-Genoud
for helpful comments.
I am also grateful to the anonymous referee for for his/her 
comments and suggestions that allowed to significantly improve
and correct this note.

\section*{Appendix 1: classification of quadratic forms on $\bbF_2^n$}

For the sake of completeness, let us give a proof
of Dickson's theorem.

(1') 
Let ${\bf q}$ be a non-degenerate quadratic form on $\bbF_2^{2k}$.
There exist coordinates on $\bbF_2^{2k}$
in which the polar bilinear form ${\bf b}$ associated to ${\bf q}$ is
written as:
$$
{\bf b}(x,y)=
x_1y_2+x_2y_1
+\cdots+
x_{2k-1}y_{2k}+x_{2k}y_{2k-1}.
$$
This is obvious since ${\bf b}$ is ``skew symmetric'' (alternating).
Indeed, one choses the first coordinate axis in an arbitrary way, then the
second axis such that ${\bf b}({\bf x}_1,{\bf y}_2)\not=0$,
and the other coordinates are in the orthogonal complement.

The quadratic form~${\bf q}$ is then as follows:
\begin{equation}
\label{ThFcn}
{\bf q}=x_1x_2+x_3x_4+\cdots+x_{2k-1}x_{2k}+
\hbox{\rm (linear terms)}.
\end{equation}
Consider one of the binary terms $x_ix_{i+1}$ of ${\bf q}$.
Changing coordinates,
$$
x_i'=x_i,
\qquad
x'_{i+1}=x_i+x_{i+1},
$$
this term is equivalent
to $x_ix_{i+1}+x_i$.
Therefore, $2$-dimensional blocks of ${\bf q}$ can be reduced to one of two types:
$$
x_ix_{i+1}
\quad\hbox{or}\quad
x_ix_{i+1}+x_i+x_{i+1}.
$$

Consider now a pair of blocks of the second type:
$$
x_ix_{i+1}+x_jx_{j+1}+x_i+x_{i+1}+x_j+x_{j+1}.
$$
The coordinate transformation:
$$
x_i'=x_i+x_j,\quad
x'_{i+1}=x_{i+1},\quad
x_j'=x_j,\quad
x'_{j+1}=x_j+x_{i+1}.
$$
sends this form to $x_ix_{i+1}+x_jx_{j+1}$.

It follows that ${\bf q}$
is equivalent to ${\bf q}_0$ if it contains an even number of blocks of the second type, 
or to ${\bf q}_1$, otherwise.
As mentioned, ${\bf q}_0$ and ${\bf q}_1$ are not equivalent.

(2')
Consider a form ${\bf q}$ of rank $2k$ on $\bbF_2^{2k+1}$.
Choose coordinates in which
the polar form~${\bf b}$ is as above, then ${\bf q}$
is again as in~(\ref{ThFcn}).
If then ${\bf q}$ contains the linear term $x_{2k+1}$,
then the coordinate transformations
$x'_{2k+1}=x_{2k+1}+x_i$
allow us to kill all other linear terms, so that ${\bf q}$ is equivalent to ${\bf q}_2$.
Otherwise, the problem is reduced to Part (1').

\section*{Appendix 2: the table of Clifford algebras}

The following well-known table of real Clifford algebras 
well illustrates Chevalley's theorem.
$$
\begin{array}{l|llllll}
p\setminus{}q&0&1&2&3&4&5\\
\hline\\
0&\R&\C&\bbH&\bbH^2&\Mat(2,\bbH)&\Mat(4,\C)\\[4pt]
1&\R^2&\Mat(2,\R)&\Mat(2,\C)&\Mat(2,\bbH)&\Mat(2,\bbH)^2&\ldots\\[4pt]
2&\Mat(2,\R)&\Mat(2,\R)^2&\Mat(4,\R)&\Mat(4,\C)&\ldots\\[4pt]
3&\Mat(2,\C)&\Mat(4,\R)&\Mat(4,\R)^2&\ldots\\[4pt]
4&\Mat(2,\bbH)&\Mat(4,\C)&\ldots\\[4pt]
5&\Mat(2,\bbH)^2&\ldots
\end{array}
$$
where the horizontal axis is parametrized by $q$ and the vertical by $p$.
The algebras with $n$ generators are represented in the diagonal $p+q=n$.

The table can be filled using the first examples:
$$
\rCl_{0,0}\simeq\R,\qquad
\rCl_{1,0}\simeq\R^2,\qquad
\rCl_{0,1}\simeq\C
$$
and the following periodicity statements:
$$
\rCl_{p,q+2}\simeq\rCl_{q,p}\otimes\bbH,\qquad
\rCl_{p+1,q+1}\simeq\rCl_{p,q}\otimes\Mat(2,\R),\qquad
\rCl_{p+2,q}\simeq\rCl_{q,p}\otimes\Mat(2,\R),
$$
where the tensor products are over $\R$.
We will not prove them, but they can be obtained directly by comparing
the generators of the algebras, see~\cite{LM,FOF},
or by analyzing quadratic forms on $\bbF_2^n$.
Let us mention that
these periodicities imply the following remarkable periodicity modulo~$8$:
$$
\rCl_{p,q+8}\simeq\rCl_{p+4,q+4}\simeq\rCl_{p+8,q}\simeq\rCl_{p,q}\otimes\Mat_{16}(\R),
$$
called the Bott periodicity.


\end{document}